\documentclass[psamsfonts,11pt]{amsart}
\usepackage{amsmath,amstext,amsthm,amssymb}

\numberwithin{equation}{section}
\theoremstyle{plain}
\newtheorem{theorem}{Theorem}
\newtheorem{lemma}{Lemma}[section]
\newtheorem{proposition}[lemma]{Proposition}
\newtheorem{corollary}[lemma]{Corollary}
\theoremstyle{definition}
\newtheorem{definition}[lemma]{Definition}
\newtheorem{remark}[lemma]{Remark}
\newtheorem*{fact}{Fact}
\newtheorem*{examples}{Examples}

\newcommand{\ball}[1]{B\left( {#1} \right)}
\newcommand{\closure}[1]{\overline{#1}}
\newcommand{\defeq}{\mathrel{\mathop:}=}
\newcommand{\diam}{\mathrm{diam}}
\newcommand{\dist}{d}

\newcommand{\dmlge}{\, d\mlge}
\newcommand{\geodflow}{\Phi}
\newcommand{\grad}{\mathrm{grad}}
\newcommand{\image}[1]{\mathrm{im}\left( {#1} \right)}
\newcommand{\interior}[1]{\mathring{#1}}
\newcommand{\limp}[1]{\lim\nolimits_p(\grad\, {#1})}
\newcommand{\limpi}[1]{\lim\nolimits_{p_i}(\grad\, {#1})}
\newcommand{\limpb}{\lim\nolimits_p(\grad\, b_\gamma)}
\newcommand{\mathreal}{\mathbb{R}}
\newcommand{\mathrealpos}{[0,\infty)}
\newcommand{\mathnatural}{\mathbb{N}}
\newcommand{\mhdf}{\mathcal{H}}
\newcommand{\mlge}{\mathcal{L}}
\newcommand{\mllb}{\nu}
\newcommand{\mlle}{\mu}
\newcommand{\mvol}[1]{\mathrm{vol}_{#1}}

\begin{document}

\title[Riemannian cylinders without conjugate points]{On the flatness of Riemannian cylinders\\ without conjugate points}
\author[Victor Bangert \& Patrick Emmerich]{Victor Bangert and Patrick Emmerich}

\begin{abstract}
What are appropriate geometric conditions ensuring that a complete Riemannian 2-cylinder without conjugate points is flat? Ex\-amples with nonpositive curvature show that one has to assume that the ends of the cylinder open sublinearly. We show that sub\-linear growth of the ends is indeed sufficient if it is measured by the length of horocycles. This is used to extend results by K. Burns and G. Knieper \cite{BKR}, and by H. Koehler \cite{HKR}, where the opening of the ends is measured in terms of shortest noncontractible loops.
\end{abstract}

\maketitle

\section{Introduction}
In 1948 E. Hopf published the following celebrated result.
\begin{theorem}[\cite{EHC}]\label{thm:hopf}
  A Riemannian metric without conjugate points on a two-dimensional torus is flat.  
\end{theorem}
This theorem and the method of its proof have attracted much interest ever since. Most importantly, by a completely different and beautiful proof, D. Burago and S. Ivanov \cite{BIT} showed in 1994 that Theorem \ref{thm:hopf} also holds for the n-dimensional torus for all $n\geq 2$. E.\ Hopf's original method is short and elegant, and has proved useful also in other situations, see e.g.\ \cite{MBC}, \cite{CKT}. It depends on the Gau\ss{}-Bonnet theorem and the invariance of the Liouville measure under the geodesic flow via an integration over the unit tangent bundle. Thus it uses the compactness of the two-torus in an essential way. If one tries to generalize this method to a noncompact manifold, one will try to apply it to an appropriate sequence of compact sets exhausting the manifold. Since the compact sets have boundaries one will be confronted with boundary terms arising from the integration. These have to be controlled in the limit. In this way, K. Burns and G. Knieper proved the following.
\begin{theorem}[\cite{BKR}]\label{thm:burnsknieper}
  Let $g$ be a complete Riemannian metric without conjugate points on the cylinder $C=S^1\times\mathreal$. Assume that
  \renewcommand{\labelenumi}{(\roman{enumi})}
  \begin{enumerate}
    \item the Gaussian curvature of $g$ is bounded below, and
    \item there exists a constant $L$ such that at every point $p\in C$ there exists a noncontractible loop of length at most $L$.
  \end{enumerate}
  Then $g$ is flat.
\end{theorem}
Under the stronger assumption that $g$ has no focal points, this had been proved by L.\ W.\ Green \cite{LGS}. It is clear that some condition of the type of condition (ii) is necessary for the result to hold: There exist complete cylinders of nonpositive Gaussian curvature (hence without conjugate points), e.g.\ surfaces of revolution in $\mathreal^3$ generated by the graph of a nonconstant function $f:\mathreal\to (0,\infty)$ with $f''\geq 0$. In these examples at least one end of the cylinder ``opens at least linearly'', i.e.\ for points $p$ in this end the length $l(p)$ of a shortest noncontractible loop at $p$ grows at least linearly with the distance from $p$ to a fixed point. In \cite{HKR} H. Koehler showed that condition (ii) in Theorem \ref{thm:burnsknieper} can be weakened to a condition that allows $l(p)$ to grow logarithmically with the distance to a fixed point.

The purpose of this paper is to prove versions of Theorem \ref{thm:burnsknieper} where both conditions (i) and (ii) are considerably relaxed. Instead of the lower bound on the Gaussian curvature $K$, we only require $K$ to be bounded below by $-t^\kappa$ for some $0<\kappa<2$, and instead of the logarithmic upper bound on $l$ we allow $l$ to grow at most like $t^\lambda$ where $\lambda>0$ and $2\lambda+\frac{\kappa}{2}<1$; here $t=\dist(\cdot,p_0)$ denotes the Riemannian distance to an arbitrarily fixed point $p_0\in C$. The precise statement is:

\setcounter{theorem}{3}
\begin{theorem}
Let $g$ be a complete Riemannian metric without conjugate points on the cylinder $C=S^1\times\mathreal$. Assume that for some constants $c,\,\kappa,\,\lambda$ in $\mathrealpos$ with $2\lambda+\frac{\kappa}{2}<1$, and for all $p\in C$ we have that
\renewcommand{\labelenumi}{(\roman{enumi}')}
\begin{enumerate}
\item the Gaussian curvature $K$ of $g$ satisfies $K(p)\geq -c\,(\dist(p,p_0)+1)^\kappa$, and
\item the length $l(p)$ of the shortest noncontractible loop at $p$ satisfies $l(p)\leq c\,(\dist(p,p_0)+1)^\lambda$.
\end{enumerate}
Then $g$ is flat.
\end{theorem}

We illustrate Theorem 4 by considering the extreme cases $\lambda=0$ or $\kappa=0$. If $l$ is assumed to be bounded, i.e.\ if we can take $\lambda=0$, we need a lower bound on the Gaussian curvature of the type $-t^\kappa$ for some $\kappa<2$. If the Gaussian curvature is bounded below, i.e.\ if we can take $\kappa=0$, we need an upper bound on $l$ of the type $t^\lambda$ for some $\lambda<\frac{1}{2}$. Finally, we note that we can replace lower curvature bounds by upper curvature bounds, cf.\ Remark \ref{rmk:upperbound}.

The improvement of Theorem \ref{thm:main} over Theorem \ref{thm:burnsknieper} is made possible by a different choice of exhaustion of $C$. While in \cite{BKR} the exhaustion is by subcylinders bounded by two geodesic loops, we use subcylinders bounded by horocycles. Here, we give a brief description of this exhaustion. The details are the content of Section 3.

Given a ray $\gamma$ in $C$, i.e.\ a (unit-speed) minimal geodesic $\gamma:\mathrealpos\to C$, we consider its Busemann function $b_\gamma:C\to\mathreal$ defined by
\[b_\gamma(p)\defeq\lim_{t\to\infty}(\dist(p,\gamma(t))-t)\]
and its horocyles
\[h^\gamma_t\defeq b_\gamma^{-1}(-t).\]
The notation is chosen so that $\gamma(t)\in h^\gamma_t$. If $C$ has no conjugate points and satisfies the condition
\begin{equation}\label{eqn:sublinear}
  \liminf_{t\to\infty}\frac{1}{t} l(\gamma(t))<1,
\end{equation}
then $b_\gamma$ is a proper function and each of its horocyles $h^\gamma_t$ is a closed curve winding once around the cylinder, see Proposition \ref{prp:regular}. Obviously, condition (ii') implies that \eqref{eqn:sublinear} is satisfied for every ray $\gamma$. To define the exhaustion we choose two rays $\gamma_1$, $\gamma_2$ converging to the two ends of $C$. For sufficiently large $t$ the horocyles $h^{\gamma_1}_t$ and $h^{\gamma_2}_t$ bound a compact subcylinder of $C$. For $t\to\infty$ these subcylinders form the exhaustion that we use in the proof of Theorem \ref{thm:main}. The fact that horocycles of a ray are equidistant makes a crucial difference in the treatment of the boundary terms arising from the integration.

As a precursor to Theorem \ref{thm:main} we prove the following result in which we replace (i') and (ii') by a sublinear bound on the lengths of horocycles. In this case we can omit any bound on the curvature.

\setcounter{theorem}{2}
\begin{theorem}
    Let $g$ be a complete Riemannian metric without conjugate points on the cylinder $C=S^1\times\mathreal$. Assume there exist two rays $\gamma_1,\, \gamma_2:\mathrealpos\to C$ converging to the different ends of $C$ such that, for $i\in\{1,2\}$, the $1$-dimensional Hausdorff measures of the corresponding horocycles $h^{\gamma_i}_t$satisfy
\[\lim_{t\to\infty}\frac{1}{t}\mhdf^1(h^{\gamma_i}_t)=0.\]
Then $g$ is flat.
\end{theorem}

\begin{examples}\label{examples}
  Cylinders of revolution with $K\leq 0$ and conical ends. We consider $C^\infty$-functions $f:\mathreal\to (0,\infty)$ with the following properties:
  \renewcommand{\labelenumi}{(\roman{enumi})}
  \begin{enumerate}
  \item $f''>0$ on some nonempty bounded open interval $I\subseteq\mathreal$.
  \item $f''|(\mathreal\setminus I)\equiv 0$.
  \item $f$ attains its minimum at some point $t_0\in I$.
  \end{enumerate}
Then the surface of revolution generated by the graph of $f$ is a nonflat complete cylinder $C_f\subseteq \mathreal^3$ with $K\leq 0$ and two conical ends. Each conical end of $C_f$ can be developed to a part of a sector in the Euclidian plane. In particular, if a ray $\gamma$ is contained in one of the conical ends of $C_f$, then the horocycles of $\gamma$ through points of $\gamma$ are geodesic and orthogonal to $\gamma$. If the ray $\gamma$ is part of a generating line of the cone, then these horocycles are geodesic loops if the angle of the sector is smaller than $\pi$, while they are noncompact complete geodesics if this angle is greater or equal to $\pi$. If this angle converges to zero, then also
\[\lim_{t\to\infty}\frac{1}{t}\mhdf^1(h^\gamma_t)\]
converges to zero. This shows that Theorem 3 is close to being optimal.

If the angle is $\pi$, and if $p_0\in C_f$ is an arbitrary point, then there is a constant $c>0$ such that
\[2\dist (p,p_0)-c\leq l(p)\leq 2 \dist (p,p_0)+c\]
for all points $p$ in the conical end. In particular we have
\[\lim_{t\to\infty}\frac{1}{t} l(\gamma(t))=2.\]
In Remark \ref{rmk:liminf2} we note that, in general, the condition
\[\liminf_{t\to\infty}\frac{1}{t}l(\gamma(t))<2\]
implies that all horocycles of $\gamma$ are compact. The preceding calculation shows that this bound is optimal.
\end{examples}

\noindent{}{\bfseries Plan of the paper.} In Section 2 we collect facts on proper, regular distance functions $f$ on general Riemannian $n$-manifolds. In particular, we prove the continuity of the $(n-1)$-dimensional Hausdorff measure $\mhdf^{n-1}(f^{-1}(t))$ of the level sets $f^{-1}(t)$ as a function of $t$, see Proposition \ref{prp:hcontinuous}.

In Section 3 we investigate Busemann functions on complete cylinders $C=S^1\times\mathreal$ without conjugate points. Under appropriate conditions on the length function $l$, we prove that they are proper and regular distance functions, see Proposition \ref{prp:regular}. This allows us to define the exhaustion function on $C$ that is used in the proof of Theorem 3, see Corollary \ref{cor:distancefunction}.

In Section 4 we derive the fundamental differential inequality \eqref{eqn:fundamental} that is the key to the rigidity results. Here we use E. Hopf's method and estimate the boundary terms as in the work \cite{BKR} by K. Burns and G. Knieper.

The differential inequality \eqref{eqn:fundamental} contains a function $h$ given by the sum of lengths of two horocycles. Section 5 is devoted to the crucial inequality \eqref{eqn:isoperimetric} between the variation of $h$ and the total curvature of the domain bounded by two horocycles. This inequality goes back to G. Bol \cite{GBI} and F. Fiala \cite{FFL}. It is proved in \cite{BZG} in a sligthly more special form. In the appendix we extend the arguments from \cite{BZG} so as to complete the proof of \eqref{eqn:isoperimetric}.

Sections 6 and 7 - 8 contain the proofs of Theorem 3 and Theorem 4, respectively.

\section{Regular distance functions on complete Riemannian manifolds}

We start by recalling some facts from critical point theory for distance functions on $n$-dimensional complete Riemannian manifolds $(M,g)$, cf.\ \cite{KGC} or \cite[Section 11.1]{PPR}.

Let $\dist$ denote the Riemannian distance on $M$, and let $A$ be a nonempty closed subset of $M$. Then the \emph{distance function} from $A$ is defined by
\[f(p)\defeq \dist(p,A)\defeq\min\{\dist(p,q):q\in A\}\]
for every $p\in M$. The function $f$ is Lipschitz with constant one. By Rade\-macher's theorem it is differentiable except possibly on a set $N_f$ of measure zero. Let 
\begin{equation}\label{eqn:limgrad}
  \limp{f}\defeq TM_p\cap\overline{\{\grad f(q):q\in M\setminus N_f\}}
\end{equation}
for every $p\in M$. Of course, this definition is meaningful for any Lipschitz function $f:M\to\mathreal$. If $f$ is the distance function from $A$, and if $p\in M\setminus A$, then the set $\limp{f}$ coincides with the set $S_{pA}$, defined by
\[ \{-\dot{\rho}(0)\;|\; \rho:[0,f(p)]\to M \text{ is a unit-speed geodesic}, \rho(0)=p,\, \rho(f(p))\in A\}. \]
To prove this, observe that along such a geodesic $\rho$ we have $f\circ \rho (t)=f(p)-t$, and hence $S_{pA}=\{\grad f (p)\}$ whenever $p\in M\setminus (A\cup N_f)$. The assertion follows, since a limit of shortest geodesic connections to $A$ is a shortest geodesic connection to $A$, and since $f$ is differentiable at inner points of such geodesics, cf.\ \cite[Theorem 3.8.2 (4)]{SST}. In particular, this consideration shows that $|\grad f|=1$ on the set $M\setminus(A\cup N_f)$.

\begin{definition}
A distance function $f$ is called \emph{regular} at a point $p\in M\setminus A$ if the set $\limp{f}$ is contained in an open hemisphere of the unit sphere in $TM_p$, i.e.\ if there exists a vector $w\in TM_p$ such that $g(v,w)>0$ for all $v\in\limp{f}$.
\end{definition}

In analogy to the case of smooth functions, versions of the implicit function theorem and of the isotopy lemma are valid for distance functions, cf.\ \cite[Propositions 1.7 and 1.8]{KGC}: If $r\in (0,\infty)$ is a regular value for $f$, i.e.\ if $f^{-1}(r)$ contains only regular points of $f$, then $f^{-1}(r)$ is an $(n-1)$-dimensional topological manifold. If $A\subseteq M$ is compact, i.e.\ if $f$ is proper, and if $[r_1,r_2]\subseteq (0,\infty)$ contains only regular values for $f$, then $f^{-1}([r_1,r_2])$ is homeomorphic to $f^{-1}(r_1)\times [r_1,r_2]$.

If $f:M\to\mathreal$ is a proper Lipschitz function, then the coarea formula \cite[13.4.2]{BZG} implies that the real function
\[t\mapsto\mhdf^{n-1}(f^{-1}(t))\]
is $\mlge^1$. In general it can have discontinuities, even if $f$ is a distance function. We will now show that we have continuity in the case of regular distance functions $f$. We could not find this simple, but potentially useful fact in the literature. Essentially, it follows from the following obvious continuity property of the generalized gradient of distance functions. Suppose $f:M\to\mathreal$ is the distance function from a closed set $A$ and the sequence $p_i\in M\setminus A$ converges to $p\in M\setminus A$. If a sequence $v_i\in\limpi{f}$ converges to $v\in TM_p$, then $v\in\limp{f}$, see \eqref{eqn:limgrad}.

\begin{proposition}\label{prp:hcontinuous}
  Let $M$ be a complete Riemannian manifold, let $\emptyset\neq A\subseteq M$ be compact, and suppose that $f=\dist(\cdot,A)$ is regular on $f^{-1}((a,b))$, where $(a,b)\subseteq(0,\infty)$. Then the function
\[t\in (a,b)\mapsto \mhdf^{n-1}(f^{-1}(t))\]
is continuous.
\end{proposition}
\begin{proof}
Obviously, it suffices to treat the case $[a,b]\subseteq(0,\infty)$, where $f$ is regular on the compact set $f^{-1}([a,b])\subseteq M$. Then \cite[p.\ 361]{KGC} implies that there exists a smooth unit vector field $X$ on $f^{-1}((a,b))$ and $\theta\in(0,\pi/2)$ such that
\[\angle(X_p,v)\leq\theta\]
whenever $p\in f^{-1}((a,b))$ and $v\in\limp{f}$. If $\gamma:(\alpha,\beta)\to f^{-1}((a,b))$ is a flow line of $X$ and $[s,t]\subseteq (\alpha,\beta)$, then
\begin{equation}\label{eqn:flowline}
  f\circ\gamma(t)- f\circ \gamma(s)\geq\cos\theta\, (t-s),
\end{equation}
cf.\ \cite[Lemma 1.5]{KGC}. The flow of $X$ allows us to realize the level sets $f^{-1}(t)$ for $t\in (a,b)$ locally as Lipschitz graphs with uniform Lipschitz constant as follows: If $p\in f^{-1}(t)$ we choose a local hypersurface $H\subseteq M$ through $p$ such that the flow $\varphi$ of $X$ induces a chart
\[\varphi:H\times (-\epsilon,\epsilon)\to V\subseteq f^{-1}((a,b))\]
of $M$ for some $\epsilon>0$. By \eqref{eqn:flowline}, for $H$ small enough, there exists $\delta>0$ such that each flow line $\varphi(x,\cdot):(-\epsilon,\epsilon)\to V$, $x\in H$, intersects $f^{-1}(s)$ in exactly one point, for every $s\in (t-\delta, t+\delta)$. Hence, for each such $s$ there is a unique function $g^s:H\to (-\epsilon,\epsilon)$ such that
\[\{\varphi(x,g^s(x)):x\in H\}=f^{-1}(s)\cap V.\]

Since $f$ is Lipschitz we can assume that $f\circ \varphi$ is Lipschitz with constant $L<\infty$ with respect to the metric
\[\tilde{\dist}((x,\sigma),(y,\tau))=\dist (x,y)+|\tau-\sigma|\]
on $H\times (-\epsilon,\epsilon)$, where $\dist$ denotes the distance on $M$ restricted to $H$. To prove the Lipschitz continuity of $g^s$, let $x,y\in H$. We can assume that $g^s(x)\geq g^s(y)$. We use \eqref{eqn:flowline} to estimate
\begin{eqnarray*}
  s=f(\varphi(x,g^s(x)))&\geq& f(\varphi(x,g^s(y))+ \cos\theta \,(g^s(x)-g^s(y))\\
                        &\geq& f(\varphi(y,g^s(y)) -L\, \dist(x,y) +\cos\theta\, (g^s(x)-g^s(y)).
\end{eqnarray*}
Since $f(\varphi(y,g^s(y))=s$, we conclude that
\[0\leq g^s(x)-g^s(y)\leq\frac{L}{\cos\theta}\,\dist(x,y).\]
Hence the functions $g^s$, $s\in(t-\delta,t+\delta)$, are all Lipschitz with constant $L/\cos\theta$. In particular, by Rademacher's Theorem, every $g^s$ is almost everywhere differentiable on $H$.

Next we assume that $g^s$ is differentiable at $x\in H$. We let
\[W=\{(w,Dg^s(x)w):w\in TH_x\}\subseteq T(H\times (-\epsilon,\epsilon))_{(x,g^s(x))}\]
 denote the tangent space of the graph of $g^s$ at $x$. We will use the following elementary fact, the proof of which is given below.
 \begin{fact}
   The hyperplane $D\varphi (W)$ in $TM_p$, $p=\varphi(x,g^s(x))$, is orthogonal to every element of $\limp{f}$.
 \end{fact}
\noindent{}So, if $g^s$ is differentiable at $x$, then $\limp{f}$, $p=\varphi(x,g^s(x))$, consists of a single vector. Otherwise $\limp{f}$ would contain two unit vectors $v_1=-v_2$, contradicting the regularity of $f$ at $p$.

Now the above-mentioned continuity property of $\limp{f}$ implies the following. If $(x_i,s_i)\in H\times (t-\epsilon,t+\epsilon)$ converge to $(x,t)\in H\times\{t\}$, if the $g^{s_i}$ are differentiable at $x_i$ and if $g^t$ is differentiable at $x$, then the differentials $Dg^{s_i}(x_i)$ converge to $Dg^t(x)$. Now the area formula \cite[3.2.3]{GMT} and Lebesgue's theorem on dominated convergence imply the following: If $\lambda$ is a continuous function with compact support in $V=\varphi(H\times(-\epsilon,\epsilon))$ and if $t_i\in(t-\epsilon,t+\epsilon)$ satisfy $\lim t_i=t$, then
\[\lim_{i\to\infty}\int_M \lambda\chi_{f^{-1}(t_i)}\, d\mhdf^{n-1}=\int_M \lambda \chi_{f^{-1}(t)}\, d\mhdf^{n-1}.\]
Finally, a partition of unity argument shows that
\[\lim_{i\to\infty}\mhdf^{n-1}(f^{-1}(t_i))=\mhdf^{n-1}(f^{-1}(t)).\]

To prove Fact, let $\gamma$ be a curve through $\gamma(0)=p$ such that $\dot{\gamma}(0)$ exists. Then the first variation formula implies
\begin{equation}\label{eqn:limpfsupport}
  \limsup_{t\downarrow 0}\frac{1}{t}(f\circ\gamma(t)-f(p))\leq \langle \dot{\gamma}(0),v\rangle
\end{equation}
for every $v\in\limp{f}$. If $\gamma$ takes values in $f^{-1}(s)$, then \eqref{eqn:limpfsupport} implies that
\[\langle\dot{\gamma}(0),v\rangle\geq 0\]
for every $v\in\limp{f}$. Replacing $\gamma$ by $\tilde{\gamma}(t)=\gamma(-t)$ we see that in fact
\[\langle\dot{\gamma}(0),v\rangle= 0\]
for every $v\in\limp{f}$. 
\end{proof}

\section{Busemann functions on cylinders without conjugate points}

In this section we define a proper function $b:C\to (-\infty,0]$, whose superlevels $H_t\defeq b^{-1}[-t,0]$ provide the exhaustion of $C$ by compact sets to which we will apply Hopf's method. In a neighborhood of each end of $C$ the function $b$ coincides with the Busemann function of a ray converging to this end. It is a helpful fact, proved in Proposition \ref{prp:regular}, that $b$ is regular in the sense of distance functions.

First we recall the definition of Busemann functions, cf.\ \cite[Section 22]{HBT} or \cite{PPR}. Let $(M,g)$ be a complete Riemannian manifold. A geodesic $\gamma:I\to M$, defined on a possibly infinite interval $I\subseteq\mathreal$, is \emph{minimal} if $\dist(\gamma(s),\gamma(t))=|t-s|$ for all $s,t\in I$. A minimal geodesic $\gamma:\mathrealpos\to M$ is called a \emph{ray}. A ray $\rho$ is called a \emph{coray} to a ray $\gamma$ if it is a limit of minimal geodesics starting at $\rho(0)$ and ending on $\gamma$. Clearly, given a ray $\gamma$ and a point $p$ in $M$, there exists a coray to $\gamma$ emanating from $p$. However, it need not be uniquely determined.

\begin{definition}[{\cite[equation (22.3)]{HBT}}]\label{defbusemann}
  Let $\gamma$ be a ray in $(M,g)$. Its \emph{Busemann function} $b_\gamma:M\to\mathreal$ is defined by
  \begin{equation}\label{eqn:defbusemann}
    b_\gamma(p) \defeq \lim\limits_{t\to\infty}(\dist(p,\gamma(t))-t).
  \end{equation}
In particular, we have $b_\gamma(\gamma(s))=\lim\limits_{t\to\infty}(|t-s|-t)=-s.$
\end{definition}

The level sets of a Busemann function are called \emph{horospheres}, in the case of a surface also \emph{horocycles}. Observe that the horosphere $b_\gamma^{-1}(a)$ is the limit of the distance spheres $\partial\ball{\gamma(t),t+a}$ as $t\to\infty$. One should think of $b_\gamma$ as a ``distance function from $\gamma(\infty)$'', as is justified by the following proposition.

\begin{proposition}[{\cite[equation (22.16)]{HBT}}, {\cite[Theorem 3.8.2]{SST}}]\label{prp:busemann}
  Let $\gamma$ be a ray in $(M,g)$, then
  \renewcommand{\labelenumi}{(\roman{enumi})}
  \begin{enumerate}
    \item 
    $b_\gamma$ is 1-Lipschitz. If $a\in b_\gamma(M)$, then 
      \[b_\gamma(p)=a+\dist(p, b_\gamma^{-1}(a)) \quad\text{for all }p\text{ with }b_\gamma(p)\geq a.\]
    
    \item
    A unit-speed geodesic $\rho:\mathrealpos\to M$ is a coray to $\gamma$ if and only if 
    \[ b_\gamma\circ\rho(t) = b_\gamma\circ\rho(0) -t \quad \text{for all } t\in\mathrealpos.\]  
    
    \item
     If $p\in M$, then
     \[\limpb = \{-\dot{\rho} (0) : \rho \text{ is a coray to } \gamma, \rho(0)=p\}.\]
  \end{enumerate}
\end{proposition}

By (i), for every $a\in\mathreal$, the Busemann function $b_\gamma$ and the distance function from the horosphere $b_\gamma^{-1}(a)$ coincide, up to the constant $a$, on the set $b_\gamma^{-1}([a,\infty))$. In particular, $b_\gamma$ is almost everywhere differentiable with unit gradient. A Busemann function is called \emph{regular} if for every $a\in\mathreal$ the function $b_\gamma|(b_\gamma^{-1}[a,\infty))$ is regular in the sense of distance functions.

\begin{definition}
  Let $(C,g)$ be a complete Riemannian cylinder. For every point $p\in C$ let $l(p)$ denote the minimal length of noncontractible loops with basepoint $p$.
\end{definition}

\begin{proposition}\label{prp:regular}
  Let $(C,g)$ be a complete Riemannian cylinder without conjugate points and let $\gamma:\mathrealpos\to C$ be a ray such that
  \begin{equation}\label{eqn:liminf}
    \liminf_{t\to\infty}\frac{1}{t}l(\gamma(t))<1.
  \end{equation}
Then $b_\gamma$ is a proper and regular Busemann function. Each of its level sets is homeo\-morphic to $S^1$ and generates the fundamental group of the cylinder.
\end{proposition}

In the proof of this proposition we will need the following general lemma and two remarks on basic properties of shortest noncontractible loops.

\begin{lemma}[{\cite[p.\ 191, Lemma 1]{DBA}}]\label{lem:bleecker}
  Every simple closed curve on a cylinder is either contractible or generates the fundamental group.
\end{lemma}

\begin{remark}\label{rmk:loops}
 Each shortest noncontractible geodesic loop on a Riemannian cylinder is simple and generates the fundamental group.
\end{remark}

To prove the remark, we first show that shortest noncontractible geodesic loops are simple. Let $\delta:[0,l(p)]\to C$ be such a loop based at a point $p\in C$. Assume $\delta$ is not simple. Then there exist $0<s_1<s_2<l(p)$ such that $\delta(s_1)=\delta(s_2)$. The restriction of $\delta$ to the set $[0,s_1]\cup[s_2,l(p)]$ is a contractible loop by minimality of $\delta$. Hence $\delta|[s_1,s_2]$ is noncontractible. This implies that $\delta|[0,s_2]*(-\delta|[0,s_1])$ is a noncontractible loop based at $p$ and of length $l(p)$ that can be shortened, a contradiction. This proves that such loops are simple. Lemma \ref{lem:bleecker} implies that they are generators of the fundamental group.

\begin{definition}\label{def:shortestloops}
  Let $\gamma$ be a ray on a complete Riemannian cylinder $C$ without conjugate points. For all $t\geq 0$, let $\delta_t:[0,l(\gamma(t))]\to C$ be the shortest noncontractible geodesic loop based at $\gamma(t)$ and having the same orientation as $S^1\times\{ 0 \}\subseteq C$.
\end{definition}

\begin{remark}[{\cite[p.\ 631, lines 13-17]{BKR}}]
  If $s < t$, then $\delta_s$ and $\delta_t$ are disjoint and bound a compact subcylinder.
\end{remark}

\begin{proof}[Proof of Proposition \ref{prp:regular}]
 (a) We first prove that on a complete simply connected surface $(S,g)$ without conjugate points every ray $\gamma:\mathrealpos\to S$ has precisely one coray starting at any given point $p\in S$: If $p\in\image{\gamma}$, then our claim is an obvious consequence of the fact that for any pair of points $p,\,q\in S$ there is precisely one geodesic from $p$ to $q$ up to parametrization. If $p\notin\image{\gamma}$, we introduce polar coordinates on $TM_p\setminus\{0_p\}$ and write
\[\exp_p^{-1}\circ\gamma\,(t) = r(t)\, (\cos \varphi(t),\sin \varphi(t)).\]
Using the uniqueness of geodesics between pairs of points again, we see that $\varphi:\mathrealpos\to\mathreal$ is injective, hence strictly monotonic. This implies the uniqueness of corays in $S$.

(b) Next we show that for every $p\in C$ there are at most two corays to $\gamma$ emanating from $p$. Let $r_i\to\infty$ be a sequence of nonnegative real numbers, and let $\rho_i:[0,L_i]\to C$ be minimal geodesics starting at $p$ and ending at $\gamma(r_i)$. Denote by $\widetilde{\rho}_i$ the lifts of the $\rho_i$ by the universal Riemannian covering $\exp_p:TC_p\to C$, starting at $\widetilde{\rho}_i(0)=0_p$. If $r_i<r_j$ then the minimal geodesics $\rho_i, \gamma|[r_i,r_j], -\rho_j$ constitute a simple closed curve on the cylinder. Lemma \ref{lem:bleecker} implies that, for all $i$, the endpoints $\widetilde{\rho}_i(L_i)$ of the lifts lie on two neighbouring lifts of the ray $\gamma$. By (a), the sequence $(\dot{\rho}_i(0))$ has at most two limit points.

(c) Here we prove the regularity of the Busemann function $b_\gamma$. Suppose $b_\gamma$ is not regular at a point $p\in C$. By (b) and Proposition \ref{prp:busemann} (iii) this would imply the existence of two $\gamma$-corays $\rho_+$ and $\rho_-$ emanating from $p$ with opposite initial vectors $(\rho_+)\dot{} (0)=-(\rho_-)\dot{}(0)$. They would compose a complete geodesic $\rho:\mathreal\to C$, $\rho|[0,\infty)=\rho_+$ that in both directions approaches the same end as the ray $\gamma$. As we will see, this contradicts assumption \eqref{eqn:liminf}.

For every $t\in[0,\infty)$ let $C_t$ denote the closed (noncompact) subcylinder that is bounded by $\delta_t$, see Definition \ref{def:shortestloops}, and that contains $\gamma(0)$. By \cite[Lemma 12]{HKR} and assumption \eqref{eqn:liminf}, we have $C_t\uparrow C$ for $t\uparrow\infty$. Since $\gamma((t,\infty))\subseteq C\setminus C_t$ for every nonnegative $t$, we see that the $\gamma$-corays $\rho_\pm$ intersect the loops $\delta_t$ for sufficiently large $t\in\mathrealpos$. Fix such $t$ and set $r_+\defeq\min\rho_+^{-1}(\image{\rho_+}\cap\image{\delta_t})$, so that $q_+\defeq\rho_+(r_+)$ is the first point of intersection of $\rho_+$ and $\delta_t$ along $\rho_+$. Let $r_-$ and $q_-$ be analogously defined for the ray $\rho_-$. By construction we have three curves starting at $q_-$ and ending at $q_+$ that are simple and pairwise disjoint up to their endpoints: the curve $\rho|[-r_-,r_+]$, and the two curves obtained by splitting the loop $\delta_t$ at the points $q_-$ and $q_+$. Since $\delta_t$ generates the fundamental group,  Lemma \ref{lem:bleecker} implies that $\rho|[-r_-,r_+]$ is homotopic to one of the other two curves. Thus it is homotopic to a curve of length less than $l(\gamma(t))$. Since $g$ has no conjugate points this implies
\begin{equation}\label{eqn:regineq1}
  r_+ + r_- < l(\gamma(t)).
\end{equation}

On the other hand, by the triangle inequality
\[
 \dist (\gamma(0),\gamma(t))\leq \dist(\gamma(0),p) + r_\pm + \dist (q_\pm,\gamma(t)),
\]
we have the following lower estimates for $r_+$ and $r_-$ respectively,
\begin{equation}\label{eqn:regineq2}
  t\leq \dist (\gamma(0),p)+r_\pm+l(\gamma(t))/2.
\end{equation}
Combining \eqref{eqn:regineq1} and \eqref{eqn:regineq2} we obtain
\[
  t\leq \dist (\gamma(0),p) + l(\gamma(t))
\]
for all sufficiently large $t\in\mathrealpos$. This contradicts assumption \eqref{eqn:liminf}.

(d) Here we show that $b_\gamma$ is proper. It suffices to show that there is $\epsilon>0$ and a sequence $t_i\to\infty$ such that for every $p\in C$ we have
\begin{equation}\label{eqn:escape}
  \dist (p,\gamma(0))\geq 2 t_i \Longrightarrow |b_\gamma(p)|\geq \epsilon t_i.
\end{equation}
Choose $l_0>l(\gamma(0))/2$ and let $C_{+}$ and $C_-\supseteq\gamma([l_0,\infty))$ be the two unbounded components of $C\setminus B(\gamma(0),l_0)$. If $p\in C_{+}$ and $t\geq l_0$, then
\[\dist (p,\gamma(t))\geq \big(\dist(p,\gamma(0))-l_0\big)+\big(t-l_0\big),\]
and hence $b_\gamma(p)\geq \dist (p,\gamma(0))-2 l_0$ by \eqref{eqn:defbusemann}.

For $p\in C_-$ we have to use assumption \eqref{eqn:liminf}, by which we can choose $0<\epsilon<1$ and a sequence $t_i\to\infty$ such that $l(\gamma(t_i))\leq 2(1-\epsilon)t_i$ for all $i$. Since $b_\gamma$ is $1$-Lipschitz and $b_\gamma(\delta_{t_i}(0))=b_\gamma(\gamma(t_i))=-t_i$, we have that on the loop $\delta_{t_i}$ the Busemann function $b_\gamma$ is bounded above by $-t_i+l(\gamma(t_i))/2\leq -\epsilon t_i$. We will now show that this implies that $b_\gamma(p)\leq -\epsilon t_i$ for all $p\in C\setminus C_{t_i}$, and hence for all $p\in C_-$ such that $\dist(p,\gamma(0))\geq 2t_i$: If $p\in C\setminus C_{t_i}$ there exists $j>i$ such that $t_j>t_i$ and $p\in C_{t_j}\setminus C_{t_i}$. This is true since $C_t\uparrow C$ for $t\uparrow\infty$, cf.\ the proof of (c). Since, by (c), $b_\gamma$ is a regular distance function, \cite[Prop.\ 1.6]{KGC} implies that $b_\gamma$ does not have any local maxima, and hence
\[b_\gamma(p)\leq \max_{\partial (C_{t_j}\setminus C_{t_i})} b_\gamma\leq \max \{-\epsilon t_i, -\epsilon t_j\}= -\epsilon t_i\]
as claimed.

(e) From the isotopy lemma for regular distance functions we conclude that for all compact intervals $[a_1,a_2]\subseteq\mathreal$ there exists a homeomorphism $h:b_\gamma^{-1}([a_1,a_2])\to b_\gamma^{-1}(a_1)\times [a_1,a_2]$; in particular, all horocycles are homeomorphic. This implies, that $C$ is homeomorphic to $b_\gamma^{-1}(a)\times\mathreal$ for any $a\in\mathreal$. Hence each $b_\gamma^{-1}(a)$ is connected and noncontractible, and thus generates the fundamental group of the cylinder.
\end{proof}

Next we describe the construction of the exhaustion functions used in this paper. We assume that there exist two rays $\gamma_1,\gamma_2:\mathrealpos\to C$, converging to the different ends and such that
\begin{equation}\label{eqn:liminfi}
  \liminf_{t\to\infty}\frac{1}{t}l(\gamma_i(t))<1 \quad\text{for }i=1,2.
\end{equation}
Our exhaustion function $b$ will depend on the choice of these rays. First note that by replacing $\gamma_1$ (or $\gamma_2$) by an appropriate subray, we may assume
\begin{equation}\label{eqn:disjoint}
  b_{\gamma_1}^{-1}((-\infty,0))\cap b_{\gamma_2}^{-1}((-\infty,0))=\emptyset.
\end{equation}
This can be seen using Proposition \ref{prp:regular}: If $t_0\defeq\min b_{\gamma_1}|b_{\gamma_2}^{-1}(0)$, then we have $b_{\gamma_1}^{-1}((-\infty,t_0))\cap b_{\gamma_2}^{-1}(0) =\emptyset$. Since $\gamma_1$ and $\gamma_2$ converge to different ends, we have $b_{\gamma_1}^{-1}((-\infty,t_0))\cap b_{\gamma_2}^{-1}((-\infty,0))=\emptyset$. Now replacing $\gamma_1$ by its subray $t\mapsto\gamma_1(t_0+t)$ yields \eqref{eqn:disjoint}.

\begin{definition}
Assuming \eqref{eqn:liminfi} and \eqref{eqn:disjoint} we define $b:C\to(-\infty,0]$ by
\begin{equation}\label{eqn:defb}
  b(p)=\left\{
\begin{array}{ll}
  b_{\gamma_1}(p) &\mathrm{if}\;b_{\gamma_1}(p)\leq 0,\\
  b_{\gamma_2}(p) &\mathrm{if}\;b_{\gamma_2}(p)\leq 0,\\
  0 &\mathrm{otherwise}.
\end{array}
\right.
\end{equation}
 We will call $b$ the \emph{Busemann function associated to the rays $\gamma_1$, $\gamma_2$} satisfying \eqref{eqn:liminfi} and \eqref{eqn:disjoint}.
\end{definition}

\begin{corollary}\label{cor:distancefunction}
 Let $(C,g)$ be a complete Riemannian cylinder without conjugate points, and assume that there exist rays $\gamma_1,\gamma_2:\mathrealpos\to C$ converging to the different ends of $C$ and satisfying \eqref{eqn:liminfi} and \eqref{eqn:disjoint}. Let $b$ be the Busemann function associated to $\gamma_1$, $\gamma_2$. Then:
 \renewcommand{\labelenumi}{(\roman{enumi})}
 \begin{enumerate}
 \item For every $t\in\mathrealpos$ the superlevel $H_t\defeq b^{-1}([-t,0])\subseteq C$ is a compact subcylinder; its boundary $h_t\defeq b^{-1}(-t)$ is the union of the two horocycles $h^{\gamma_1}_t\defeq b_{\gamma_1}^{-1}(-t)$ and $h^{\gamma_2}_t\defeq b_{\gamma_2}^{-1}(-t)$.
  \item For all $t\in\mathrealpos$ and all $p\in H_t\setminus H_0$ we have $b(p)=\dist (p,h_t) -t$. The function $b$ is regular in the sense of distance functions on the set $C\setminus H_0$.
\end{enumerate}
The notation is chosen so that $\gamma_i(t)\in h_t$ for all $t\in\mathrealpos$ and $i=1,2$.
\end{corollary}

\begin{proof}
The statement is an immediate consequence of Proposition \ref{prp:regular} and Proposition \ref{prp:busemann} (i).
\end{proof}

\begin{remark}\label{rmk:liminf2}
  Actually, the construction of an exhaustion function $b$ as above is possible under the following weaker condition \eqref{eqn:liminf2} on the two rays $\gamma_1$ and $\gamma_2$,
  \begin{equation}\tag{3.6'}\label{eqn:liminf2}
    \liminf_{t\to\infty}\frac{1}{t}l(\gamma_i(t))<2\quad\text{for }i=1,2.
  \end{equation}
\end{remark}

This condition \eqref{eqn:liminf2} is sharp since the examples in Section 1 show that horocycles of a ray on $C$ can be noncompact if \eqref{eqn:liminf2} is not satisfied for this ray. However, when we apply Corollary \ref{cor:distancefunction} in the proofs of Theorem \ref{thm:premain} and \ref{thm:main}, we even have the assumption that $\lim_{t\to\infty}\frac{1}{t}l(\gamma(t))=0$. So, we do not need the sharp version.

The proof of this sharp version is analogous to the proof of Proposition \ref{prp:regular} with the following additional remark concerning part (c): If $p=\rho(0)$ is in an appropriate neighborhood of $\gamma(\infty)$, then inequality \eqref{eqn:regineq1} can be improved to the equality
\begin{equation}\label{eqn:regineq1strich}\tag{3.3'}
  r_++r_-+\dist (q_+,\gamma(t))+\dist (q_-,\gamma(t))=l(\gamma(t)).
\end{equation}

\section{The fundamental inequality}
We consider a complete cylinder $(C,g)$ without conjugate points that admits two rays $\gamma_1$, $\gamma_2$ converging to the different ends of $C$ and satisfying \eqref{eqn:liminfi} and \eqref{eqn:disjoint}. Let $b:C\to(-\infty,0]$ be the Busemann function associated to $\gamma_1$ and $\gamma_2$ defined by \eqref{eqn:defb}. We apply Hopf's method to the compact superlevels $H_t=b^{-1}([-t,0])$ and estimate the boundary terms as in \cite[1.3]{BKR}. This leads to our fundamental inequality \eqref{eqn:fundamental}.

In the following $\sigma:T^1M\to M$ denotes the unit tangent bundle of a Riemannian manifold $M$. Hopf's method is based on the following observation made in \cite{EHC}.

\begin{lemma}[\cite{EHC}]
Let $(S,g)$ be a complete surface without conjugate points. Then there exists a Borel measurable and locally bounded function $u:T^1S\times\mathreal\to\mathreal$ with the following properties: For every $v\in T^1S$ the function $u_v(t)\defeq u(v,t)$ is a solution of the Riccati equation along the geodesic $\gamma_v$ with $\dot{\gamma}_v(0)=v$, i.e.\
\begin{equation}\label{eqn:riccati}
(u_v)'(t)+(u_v)^2(t)+K\circ \gamma_v(t)=0
\end{equation}
for all $t\in\mathreal$. The function $u$ is invariant under the geodesic flow $\geodflow$, i.e.\
\begin{equation}\label{eqn:invariance}
u(v,s+t)=u(\geodflow^sv,t)
\end{equation}
for all $v\in T^1S$ and all $s,t\in\mathreal.$  
\end{lemma}

\begin{proof}
  Except for the local boundedness this follows from \cite{EHC}. In \cite{EHC} Hopf proves that $u$ satisfies the uniform bound
\[|u(v,t)|\leq A\]
provided the Gaussian curvature satisfies $K>-A^2$ for some constant $A>0$ everywhere. In our case, for every relatively compact set $V\subseteq T^1S$ we can find a constant $A>0$ such that $K(\gamma_v(t))>-A^2$ for all $(v,t)\in V\times [-1,1]$. Under these assumptions a slight variation of Hopf's argument proves that
\[|u(v,0)|\leq A\,\frac{\cosh A}{\sinh A}<A+1\]
for all $v\in V$. This together with \eqref{eqn:invariance} implies the local boundedness of $u$.
\end{proof}

Let $U=u(\cdot,0):T^1C\to\mathreal$. We intend to integrate the function $U^2$ over $\sigma^{-1}(H_t)$ with respect to the Liouville measure $\mlle$ on $T^1C$. The Liouville measure $\mlle$ is invariant under the geodesic flow, and $\mlle$ is the product of the Lebesgue measure on the fibers of $T^1C$ with the Riemannian volume $\mvol{2}$ on $C$. To treat the boundary terms we use the measure $\mllb$ on $T^1C$ that is the product of the Lebesgue measure on the fibers of $T^1C$ with the one-dimensional Hausdorff measure $\mhdf^1$ on $C$. 

The following lemma is a version of \cite[1.3]{BKR}. The fact that horospheres are equidistant simplifies the proof of our case. On the other hand, the boundary $h_t$ of $H_t$ may be less regular than assumed in \cite{BKR}.

\begin{lemma}[cf.\ {\cite[1.3]{BKR}}]\label{lem:keylemma}
Let $(C,g)$ be a complete Riemannian cylinder without conjugate points and let $b:C\to(-\infty,0]$ be the Busemann function associated to the two rays $\gamma_1$, $\gamma_2$. Then we have 
   \begin{equation}\label{eqn:keylemma}
    \int_{\sigma^{-1}(H_t)} U^2 \,d\mlle \;\leq \; -2\pi \int_{H_t}K\,d\mvol{2} \;+\; 2 \int_{\sigma^{-1}(h_t)} \left|U\right|\,d\mllb
  \end{equation}
for almost every $t\in\mathrealpos$.
\end{lemma}

\begin{proof}
 Let $t\in\mathrealpos$. Integrating the Riccati equation \eqref{eqn:riccati}, we get
\[\int_{\sigma^{-1}(H_t)} \left\{\frac{1}{s}\int_0^s  u'(v,s') + u^2(v,s')+K(\gamma_v({s'}))\, d{s'} \right\}\, d\mlle(v) =0.\]
for all $s>0$. We let $s\to 0$ to obtain
\begin{eqnarray}\label{eqn:keylproof}
\int_{\sigma^{-1}(H_t)} U^2(v) \,d\mlle(v) \; \leq &-& \int_{\sigma^{-1}(H_t)} K(\sigma (v))\, d\mlle(v) \nonumber \\
&+& \limsup\limits_{s\to 0} \left| \int_{\sigma^{-1}(H_t)} \left\{\frac{1}{s}\int_0^s u'(v, s')\, d{s'} \right\} d\mlle(v)\right|.
\end{eqnarray}
By \eqref{eqn:invariance} and the $\geodflow$-invariance of $\mu$ we have
\begin{eqnarray*}
  \int_{\sigma^{-1} (H_t)}\left\{ \int_0^s u'(v,s')\, d{s'}\right\} d\mlle (v)&=&\int_{\sigma^{-1} (H_t)} \{U(\geodflow^s v)-U(v)\}\, d\mlle(v)\\
  &=&\int_{\geodflow^s(\sigma^{-1} (H_t))}U \, d\mlle - \int_{\sigma^{-1} (H_t)} U \, d\mlle.
\end{eqnarray*}
For the symmetric difference of the domains of the last two integrals we have
\begin{eqnarray*}
  \geodflow^s(\sigma^{-1} (H_t)) \;\Delta\; \sigma^{-1} (H_t)
  &\subseteq& \Phi(\sigma^{-1}(h_t)\times[-s,s])\\
  &\subseteq& \sigma^{-1}(\{p\in C : \dist (p, h_t)\leq s \})\\
  &\subseteq& (b\circ \sigma)^{-1}([t-s,t+s]),
\end{eqnarray*}
where the last inclusion holds since $b$ is $1$-Lipschitz. This implies
\begin{eqnarray*}
  \left| \frac{1}{s}\int_{\sigma^{-1} (H_t)} \left\{ \int_0^s u'(v, s')\, ds' \right\}\, d\mlle (v)\right|
 \leq \frac{1}{s}\int_{(b\circ\sigma)^{-1}([t-s,t+s])} |U(v)| \, d\mlle(v)\\
  = \frac{1}{s}\int_{t-s}^{t+s} d{s'} \left\{\int_{\sigma^{-1}(h_{s'})} |U(v)| \,d\mllb(v)\right\},
\end{eqnarray*}
where the last equality follows from the coarea formula \cite[13.4.6]{BZG} for the function $b$, which has unit gradient almost everywhere in $C\setminus H_0$, see Corollary \ref{cor:diffinequality}. Since
\[\lim\limits_{s\to 0} \frac{1}{s}\int_{t}^{t+s} d{s'} \left\{\int_{\sigma^{-1}(h_{s'})} |U| \, d\mllb\right\} = \int_{\sigma^{-1}(h_t)} |U| \, d\mllb\]
for almost every $t\in\mathrealpos$ by \cite[VII.4.14 (Hauptsatz)]{JEM}, the conclusion follows from inequality \eqref{eqn:keylproof}.  
\end{proof}

\begin{definition}\label{def:bfunctions}
  Let $b$ be the Busemann function associated to the two rays $\gamma_1$, $\gamma_2$. For all $t\in\mathrealpos$ denote the length of the boundary $h_t$ of $H_t$ by
\begin{equation}\label{eqn:hsum}
   h(t)\defeq \mhdf^1\left(h_t\right)=\mhdf^1(h^{\gamma_1}_t)+\mhdf^1(h^{\gamma_2}_t),
\end{equation}
the volume of $H_t$ by $v(t)\defeq \mvol{2}(H_t)$, the total curvature of $H_t$ by $\omega(t)\defeq \int_{H_t}K\,d\mvol{2}$, and the $U^2$-integral over $\sigma^{-1}(H_t)$ by $F(t)\defeq \int_{\sigma^{-1}(H_t)}U^2 \, d\mlle$.
\end{definition}

\begin{corollary}\label{cor:diffinequality}
  The function $h:\mathrealpos\to\mathrealpos$ is continuous. The functions $v:\mathrealpos\to\mathrealpos$, $\omega:\mathrealpos\to\mathreal$ and $F:\mathrealpos\to\mathrealpos$ are absolutely continuous and differentiable almost everywhere. For all $t\in\mathrealpos$
  \begin{eqnarray*}
  v(t)&=& v(0) + \int_0^t h({s}) \, d{s},\\
  \omega(t)&=& \omega(0)+\int_0^t d{s} \left\{ \int_{h_{s}} K\, d\mhdf^1 \right\},\\
  F(t)&=& F(0) + \int_0^t d{s}\left\{\int_{\sigma^{-1}(h_{s})}U^2 \, d\mllb \right\}.
\end{eqnarray*}
The integrands on the right hand side are almost everywhere the derivatives of the functions $v$, $\omega$, $F$, respectively. With these notions inequality \eqref{eqn:keylemma} implies the following differential inequality which is valid almost everywhere:
\begin{equation}\label{eqn:fundamental}
                            F \leq -2\pi \omega + 2 (2\pi F'h)^{1/2}.
\end{equation}
\end{corollary}

\begin{proof}
  The continuity of $h$ is a consequence of Proposition \ref{prp:hcontinuous} and the regularity of $b$, see Corollary \ref{cor:distancefunction}. The first equality follows from the coarea formula \cite[13.4.2]{BZG}, since, on the set $C\setminus H_0$, the function $b$ is almost everywhere differentiable with unit gradient. The second and third equality follow from \cite[13.4.6]{BZG}, which is a corollary of the coarea formula, for the integrable functions $K:M\to\mathreal$ and $U^2:T^1M\to\mathreal$. In particular, $s\mapsto\int_{h_{s}} K \, d\mhdf^1$ and $s\mapsto\int_{\sigma^{-1}(h_{s})}U^2 \, d\mllb$ are in $\mlge^1$. The remaining properties of the functions $v$, $\omega$ and $F$ now follow from \cite[VII.4.14 (Hauptsatz)]{JEM}. Finally, applying the Cauchy-Schwarz inequality to the last term of inequality \eqref{eqn:keylemma} we see that \eqref{eqn:keylemma} implies our fundamental inequality \eqref{eqn:fundamental}.
\end{proof}

\section{The lengths of horocycles and total curvature}

If $b:C\to (-\infty,0]$ is a Busemann function associated to two rays $\gamma_1$, $\gamma_2$ and if $b$ is smooth on $C\setminus H_0$, then the length $h(t)$ of the boundary $h_t$ of $H_t=b^{-1}([-t,0])$ satisfies
\[h'(t)=\int_{h_t} k \, d\mhdf^1,\]
where $k$ denotes the geodesic curvature of $h_t=\partial H_t$ with respect to the inward pointing normal. Using the Gauss-Bonnet formula we obtain
\[h'(t)=-\omega(t)\]
and
\[h(t_2)-h(t_1)=-\int_{t_1}^{t_2}\omega(t)\, dt\]
for all $t_1\leq t_2$ in $\mathrealpos$. If $b$ is not smooth, then this equality will in general only be an inequality, as stated in the following lemma.
\begin{lemma}\label{lem:isoperimetric}
  Let $b$ be the Busemann function associated to the two rays $\gamma_1,\,\gamma_2$. Then we have
\begin{equation}\label{eqn:isoperimetric}
  h(t_2)-h(t_1)\geq -\int_{t_1}^{t_2}\omega(t)\,dt
\end{equation}
 for all $t_1\leq t_2$ in $\mathrealpos$.
\end{lemma}

This phenomenon can already be observed in the Euclidian plane if one looks at the inner equidistants of a polygon. The effect of an inner angle $\alpha\in(0,\pi)$ on the derivative of the lengths of the inner equidistants is $2 \tan \frac{\pi-\alpha}{2}$, while it adds $\pi-\alpha$ to the geodesic curvature.

In our proof of Theorem \ref{thm:premain}, inequality \eqref{eqn:isoperimetric} is crucial to relate the curvature term $\omega(t)$ in the fundamental inequality \eqref{eqn:fundamental}, to the derivative $h'(t)$ of the lengths of horocycles. Inequality \eqref{eqn:isoperimetric} has a long history. It has been used in the investigation of isoperimetric inequalities on surfaces, see \cite{GBI}, \cite{FFL},\cite{PLL}.

Inequality \eqref{eqn:isoperimetric} can be proved by approximation with polyhedral metrics, see the book \cite{AZI} by A.~D. Aleksandrov and V.~A. Zalgaller. Yu.~D. Burago and V.~A. Zalgaller summarize this method in \cite[\S\S{} 2 - 3]{BZG} in order to use it in a proof of the isoperimetric inequality for surfaces by the method of inner equidistants. Up to a final limit argument our Lemma \ref{lem:isoperimetric} follows from Lemma \cite[3.2.3]{BZG}. This limit argument is provided in the appendix.

\section{Proof of Theorem 3}

In this section we prove

\setcounter{theorem}{2}
\begin{theorem}\label{thm:premain}
    Let $g$ be a complete Riemannian metric without conjugate points on the cylinder $C=S^1\times\mathreal$. Assume there exist two rays $\gamma_1,\, \gamma_2:\mathrealpos\to C$ converging to the different ends of $C$ such that, for $i\in\{1,2\}$, the $1$-dimensional Hausdorff measures of the corresponding horocycles $h^{\gamma_i}_t$satisfy
\[\lim_{t\to\infty}\frac{1}{t}\mhdf^1(h^{\gamma_i}_t)=0.\]
Then $g$ is flat.
\end{theorem}

We intend to use the Busemann function associated to the two rays $\gamma_1,\,\gamma_2$, cf.\ Section 2. For this we need to know that condition \eqref{eqn:liminfi} is satisfied, i.e.\ that $\liminf\frac{1}{t}l(\gamma_i(t))<1$. The following two lemmas prove that $\lim\frac{1}{t}\mhdf^1(h^{\gamma_i}_t)=0$ even implies $\lim\frac{1}{t}l(\gamma_i(t))=0$.

\begin{lemma}\label{lem:areaformula}
  Let $E$ be an unbounded, connected, open subset of $C$. Assume that there exists an unbounded component $U$ of $C\setminus\closure{E}$ and $p\in \closure{U}\cap\closure{E}$. Then we have
\[\mhdf^1(\partial E)\geq l(p).\]
\end{lemma}
\begin{proof}
  Since $U$ and $E$ are connected and unbounded and $p\in\closure{U}\cap\closure{E}$, we conclude that, for every $r>0$, both $U$ and $E$ intersect the geodesic sphere $S_p(r)=\{q\in C : \dist (p,q)=r\}$. For $r<l(p)/2$, the geodesic sphere $S_p(r)$ is diffeomorphic to a circle. Hence, for $r\in(0,l(p)/2)$, $\partial E$ has at least two points in common with $S_p(r)$. Since the distance function from $p$ is Lipschitz with constant one, we can use \cite[2.10.11]{GMT} to conclude that
\[\mhdf^1(\partial E)\geq 2 (l(p)/2)=l(p).\]
\end{proof}

\begin{lemma}
  Let $\gamma:\mathrealpos\to C$ be a ray with Busemann function $b_\gamma$ and horocycles $h^\gamma_t\defeq b_\gamma^{-1}(-t)$. Then we have
\[l(\gamma(t))\leq\mhdf^1(h^\gamma_t)\]
for every $t>l(\gamma(0))/2$.
\end{lemma}

\begin{proof}
  For fixed $t>l(\gamma(0))$ we consider the horoball
\[E^\gamma_t=b^{-1}_\gamma((-\infty,-t)).\]
Then $h^\gamma_t=\partial E^\gamma_t$. If $\gamma(t)$ can be joined to infinity by a curve $\widetilde{\gamma}:\mathrealpos\to C$ with $\widetilde{\gamma}(0)=\gamma(t)$ and $\widetilde{\gamma}(s)\notin \closure{E}^\gamma_t$ for every $s>0$, then we can apply Lemma \ref{lem:areaformula} to $E=E^\gamma_t$ and $p=\gamma(t)$, and obtain
\[\mhdf^1(h^\gamma_t)=\mhdf^1(\partial E^\gamma_t)\geq l(\gamma(t)),\]
as claimed. It remains to construct such a curve $\widetilde{\gamma}$. First note that the segment $\gamma|[0,t]$ of the ray $\gamma$ satisfies
\[b_\gamma(\gamma(s))=-s>-t,\text{ and hence }\gamma(s)\notin\closure{E}^\gamma_t,\]
for every $s\in[0,t)$. So, it suffices to find a curve in $C\setminus\closure{E}^\gamma_t$ joining $\gamma(0)$ to infinity. Let $\delta_0$ denote the noncontractible geodesic loop at $\gamma(0)$ of length $l(\gamma(0))$. Since $t>l(\gamma(0))/2$, $b_\gamma(\gamma(0))=0$ and $b_\gamma$ is $1$-Lipschitz, we see that $\delta_0\subseteq C\setminus\closure{E}^\gamma_t$. The set $C\setminus\delta_0$ has two connected components, both unbounded in $C$, and the connected set $\closure{E}^\gamma_t$ is contained in one of them. Now $\gamma(0)\in\delta_0$ can be joined to infinity by a curve in the other component, so in particular in $C\setminus\closure{E}^\gamma_t$.
\end{proof}

The preceding lemma shows that our assumption $\lim\frac{1}{t}\mhdf^1(h^{\gamma_i}_t)=0$ im\-plies $\lim\frac{1}{t}l(\gamma_i(t))=0$. In particular, condition \eqref{eqn:liminfi} is satisfied. Moreover, replacing $\gamma_1$ by a subray of $\gamma_1$, we can assume that \eqref{eqn:disjoint} holds. By Corollary \ref{cor:distancefunction} there exists the Busemann function $b$ associated to the rays $\gamma_1,\,\gamma_2$, as defined in \eqref{eqn:defb}. In the following proof we use the functions $h(t)$, $v(t)$, $\omega(t)$ and $F(t)$ defined for this function $b$, see Definition \ref{def:bfunctions}.

\begin{proof}[Proof of Theorem \ref{thm:premain}]
Suppose $F$ does not vanish identically, so that $F(t_0)>0$ for some $t_0\in[0,\infty)$. We may assume $t_0=0$ by replacing $\gamma_1,\,\gamma_2$ by appropriate subrays, and hence $F>0$ by monotonicity of $F$. From the fundamental inequality \eqref{eqn:fundamental} we infer
\[
  1+\frac{2\pi\omega}{F} \leq \sqrt{8\pi} \left ( \frac{F'}{F^2} h \right )^{1/2},
\]
and by integration, using the Cauchy-Schwarz-inequality, for every $t\in\mathrealpos$
\[
  \int_0^t \left( 1+\frac{2\pi\omega}{F}\right) \dmlge \leq \sqrt{8\pi} \left( \int_0^t \frac{F'}{F^2} \dmlge \right)^{1/2} \left( \int_0^t h \dmlge\right)^{1/2}.
\]
Observing $\frac{F'}{F^2}=\left( -\frac{1}{F}\right )'$ and using \cite[VII.4.14 (Hauptsatz)]{JEM} for the absolutely continuous function $\frac{1}{F}$, and the first equality in Corollary \ref{cor:diffinequality}, we obtain
\begin{equation}\label{eqn:leftside}
  \int_0^t \left( 1+\frac{2\pi\omega}{F}\right ) \dmlge \leq \left(\frac{8\pi}{F(0)}\right)^{1/2}v(t)^{1/2}.
\end{equation}

Next, we wish to show that we have
\begin{equation}\label{eqn:rightside}
  \int_0^t\frac{\omega}{F}\dmlge \geq -\frac{1}{F(0)} \max_{[0,t]}h
\end{equation}
for every $t\in[0,\infty)$. Let $\Omega(s)\defeq\int_0^s \omega \,d\mlge$ for every $0\leq s\leq t$. Integration by parts \cite[VII.4.16]{JEM} for the absolutely continuous functions $\Omega,\frac{1}{F}:[0,t]\to\mathreal$ implies
\begin{equation}\label{eqn:partint}
\int_0^t\frac{\omega}{F}\, d\mlge =\frac{\Omega(t)}{F(t)}+\int_0^t\Omega\left(-\frac{1}{F}\right)' d\mlge.
\end{equation}
By Lemma \ref{lem:isoperimetric} for $t_1=0$, $t_2=s$, we have $\Omega(s)\geq -h(s)\geq -\max_{[0,t]}h$ for every $0\leq s\leq t$. Moreover, by monotonicity of $F$ we have $\left(-\frac{1}{F}\right)'\geq 0$ almost everywhere. Estimating the right hand side of \eqref{eqn:partint} using these properties and \cite[VII.4.14 (Hauptsatz)]{JEM} proves \eqref{eqn:rightside}.

Combining inequalities \eqref{eqn:leftside} and \eqref{eqn:rightside} we conclude that
\begin{equation}\label{eqn:contradvh}
  \left( \frac{8\pi}{F(0)} \right)^{1/2} v(t)^{1/2}\geq t-\frac{2\pi}{F(0)} \max_{[0,t]}h
\end{equation}
for every $t\in\mathrealpos$.

But, by \eqref{eqn:hsum} and our assumption, we have $\lim_{t\to\infty} h(t)/t=0$, and thus $\lim_{t\to\infty} v(t)/t^2=0$. Accordingly, inequality \eqref{eqn:contradvh} cannot hold. Hence $F=0$. The remainder of the proof is the same as in E. Hopf's original argument: By Definition \ref{def:bfunctions}, $F=0$ implies that $U=0$ $\mlle$-almost everywhere, and then \eqref{eqn:riccati} and \eqref{eqn:invariance} imply that $K=0$, i.e.\ $g$ is flat.
\end{proof}

\section{An application of Santalo's formula}
If we want to deduce Theorem \ref{thm:main} from Theorem \ref{thm:premain} we have to convert information on the lengths of shortest noncontractible loops into information on the lengths of horocycles. This is achieved via the area estimate in Proposition \ref{prp:santalo2} below. This estimate is a consequence of Santalo's formula, cf.\ \cite{VBT}.

\begin{proposition}\label{prp:santalo1}
  Let $M$ be an $m$-dimensional Riemannian manifold, $m\geq 2$, and let $A\subseteq M$ be a subset that is the closure of its interior and has strong Lipschitz boundary. Assume that there exists $T>0$ such that the following is true: For every geodesic $\gamma$ with $\gamma(0)\in A$ there exists $t\in (0,T)$ such that $\gamma(t)\notin A$. Then we have
\[\mvol{m}(A)\leq c_m\,\mvol{m-1}(\partial A)\,T,\]
where $c_m=\frac{\alpha_{m-2}}{(m-1)\alpha_{m-1}}$ and $\alpha_m$ is the volume of the $m$-dimensional unit sphere in Euclidian space. If every geodesic $\gamma:[a,b]\to M$ with $\gamma([a,b])\subseteq A$ is minimal, then $\mvol{m}(A)\leq c_m\mvol{m-1}(\partial A)\,\diam (A)$.
\end{proposition}

\begin{proof}
  Let $\partial\mathfrak{A}^+\subseteq T^1M$ denote the set of all vectors $v\in T^1M$ such that $\sigma(v)\in\partial A$ and the geodesic $\gamma_v$ with initial vector $v$ satisfies $\gamma_v(t)\in \interior{A}$ for all sufficiently small $t>0$.
Let $\geodflow$ denote the geodesic flow of $M$. Then Santalo's formula implies
\begin{equation}\label{eqn:santalo1}
  \mlle(\geodflow(\partial\mathfrak{A}^+\times[0,T]))\leq \frac{\alpha_{m-2}}{m-1}\,\mvol{m-1}(\partial A)\,T,
\end{equation}
where $\mlle$ denotes the Liouville measure on $T^1M$, cf.\ \cite[Section 3]{VBT}. Now we use our assumption on $T$ to show that
\begin{equation}\label{eqn:santalo2}
  T^1\interior{A}\subseteq \geodflow(\partial\mathfrak{A}^+\times(0,T)).
\end{equation}
Namely, if $v\in T^1\interior{A}$ then there exists $t\in(0,T)$ such that $\gamma_v((-t,0))\subseteq\interior{A}$ and $\gamma_v(-t)\in\partial A$. This implies that $w=\dot{\gamma}_v(-t)=\geodflow(v,-t)\in\partial\mathfrak{A}^+$, so that $v=\geodflow (w,t)$ with $w\in\partial\mathfrak{A}^+$ and $t\in(0,T)$.

Finally, by the definition of the Liouville measure $\mlle$, we have
\begin{equation}\label{eqn:santalo3}
  \mlle(T^1\interior{A})=\alpha_{m-1}\mvol{m}(\interior{A}).
\end{equation}
Now our claim follows from \eqref{eqn:santalo1} - \eqref{eqn:santalo3} and the fact that $\mvol{m}(\interior{A})=\mvol{m}(A)$ since $\partial A$ is an $(m-1)$-dimensional strong Lipschitz submanifold. Finally, if every geodesic in $A$ is minimal, then our assumption is satisfied for every $T>\diam (A)$. Hence the second implication in our claim follows from the first one.
\end{proof}
\begin{proposition}\label{prp:santalo2}
  Let $g$ be a complete Riemannian metric without conjugate points on the cylinder $C=S^1\times\mathreal$. Let $D$ be a compact subcylinder of $C$ bounded by two noncontractible, piecewise regular $C^1$-curves $\Gamma_1$, $\Gamma_2$ of lengths $L_1$, $L_2$, respectively. Set $\overline{L}=\frac{1}{2}(L_1+L_2)$. Then
\[\mvol{2}(D)\leq\frac{8}{\pi}\,(\dist (\Gamma_1,\Gamma_2)+\overline{L})\,\overline{L}.\]
\end{proposition}
\begin{proof}
  We let $\Gamma$ denote a shortest geodesic segment joining $\Gamma_1$ and $\Gamma_2$, so that $L(\Gamma)=\dist(\Gamma_1,\Gamma_2)$. We consider a connected component $\widetilde{D}$ of the preimage of $D$ in the universal Riemannian covering $\widetilde{C}\simeq\mathreal^2$ of $C$. The two boundary components of $\widetilde{D}$ are lifts $\widetilde{\Gamma}_1$ of $\Gamma_1$ and $\widetilde{\Gamma}_2$ of $\Gamma_2$. We let $\widetilde{\Gamma}^n\subseteq\widetilde{D}$ denote the lifts of $\Gamma$ to $\widetilde{D}$, numbered in their natural order. For every $n\in\mathnatural$ we consider the compact subdomain $\widetilde{D}^n$ of $\widetilde{D}$ between $\widetilde{\Gamma}^0$ and $\widetilde{\Gamma}^n$. Note that
\[\mvol{2}(\widetilde{D}^n)=n\,\mvol{2}(D).\]
We will now apply Proposition \ref{prp:santalo1} to the set $\widetilde{D}^n$. Note that $\widetilde{D}^n$ is bounded by $\widetilde{\Gamma}^0$ and $\widetilde{\Gamma}^n$, both of length $L(\Gamma)$, and by two segments on $\widetilde{\Gamma}_1$ and $\widetilde{\Gamma}_2$ of lengths $nL_1$ and $nL_2$, respectively. So the total length of $\partial\widetilde{D}^n$ is $2(L(\Gamma)+n\overline{L})$. Since every geodesic in $\widetilde{C}$ is minimal, we conclude that
\[\diam (\widetilde{D}^n)\leq L(\Gamma)+n\overline{L}.\]
Now Proposition \ref{prp:santalo1} implies
\[\mvol{2}(\widetilde{D}^n)\leq\frac{2}{\pi}(L(\Gamma)+n\overline{L})^2.\]
We choose $n\in\mathnatural$ such that
\[(n-1)\overline{L}\leq L(\Gamma)<n\overline{L}.\]
Then the preceding inequalities imply
\[\mvol{2}(\widetilde{D}^n)\leq \frac{8n^2}{\pi}\overline{L}^2.\]
Since $\mvol{2}(D)=\frac{1}{n}\mvol{2}(\widetilde{D}^n)$ we conclude that
\[\mvol{2}(D)\leq \frac{8}{\pi}n\overline{L}^2\leq \frac{8}{\pi} (L(\Gamma)+\overline{L})\overline{L}.\]
\end{proof}

\begin{remark}
(1) The inequalities in Proposition \ref{prp:santalo1} are equalities if $A$ is a hemisphere. The inequality in Proposition \ref{prp:santalo2} is always strict.
\newline\noindent (2) Besicovitch's Lemma \cite{ABO} implies the following reverse inequality which is true without the assumption that there do not exist conjugate points. If $D= S^1\times[0,1]$ is a compact Riemannian cylinder with boundary components $\Gamma_1$ and $\Gamma_2$, and if $L$ denotes the minimal length of a noncontractible closed curve in $D$, then
\[\mvol{2}(D)\geq\dist (\Gamma_1,\Gamma_2)L.\]
This inequality is an equality if $D$ is flat with geodesic boundary.
\end{remark}

\section{Proof of Theorem 4}

In this section we prove

\setcounter{theorem}{3}
\begin{theorem}\label{thm:main}
Let $g$ be a complete Riemannian metric without conjugate points on the cylinder $C=S^1\times\mathreal$. Assume that for some constants $c,\,\kappa,\,\lambda$ in $\mathrealpos$ with $2\lambda+\frac{\kappa}{2}<1$, and for all $p\in C$ we have that
\renewcommand{\labelenumi}{(\roman{enumi}')}
\begin{enumerate}
\item the Gaussian curvature $K$ of $g$ satisfies $K(p)\geq -c\,(\dist(p,p_0)+1)^\kappa$, and
\item the length $l(p)$ of the shortest noncontractible loop at $p$ satisfies $l(p)\leq c\,(\dist(p,p_0)+1)^\lambda$.
\end{enumerate}
Then $g$ is flat.
\end{theorem}

\begin{remark}\label{rmk:upperbound}
  The same conclusion holds if we replace condition (i') by 
  \renewcommand{\labelenumi}{(\roman{enumi}'')}
  \begin{enumerate}
    \item \emph{the Gaussian curvature $K$ of $g$ satisfies $K(p)\leq c\,(\dist(p,p_0)+1)^\kappa$}.
  \end{enumerate}
\end{remark}

We will reduce Theorem \ref{thm:main} to Theorem \ref{thm:premain}. So, we have to prove that (i') and (ii') imply the existence of two rays $\gamma_1$, $\gamma_2$ converging to the different ends of $C$ such that the lengths $\mhdf^1(h^{\gamma_i}_t)$ of the horocycles $h^{\gamma_i}_t$ corresponding to $\gamma_i$ satisfy
\[\lim_{t\to\infty}\frac{1}{t}\mhdf^1(h^{\gamma_i}_t)=0\]
for $i=1,2$. Actually, we are going to prove that this is true for any two rays $\gamma_1$, $\gamma_2$ converging to the different ends of $C$.

The idea of the proof is as follows. By (ii') we can assume that \eqref{eqn:liminfi} and \eqref{eqn:disjoint} are satisfied for $\gamma_1$ and $\gamma_2$, and, as before, we let $b:C\to\mathrealpos$ denote the Busemann function associated  to the rays $\gamma_1$, $\gamma_2$. Similarly, we can assume that for all $t\in\mathrealpos$ the shortest noncontractible loops at $\gamma_1(t)$ and $\gamma_2(t)$ are disjoint and bound a compact subcylinder $\Delta_t\subseteq C$, cf.\ Remark \ref{rmk:loops}. Using Proposition \ref{prp:santalo2} and (ii') we obtain an estimate for $\mvol{2}(\Delta_{t_2}\setminus\Delta_{t_1})$ of the type
\begin{equation}\label{eqn:volumeDelta}
  \mvol{2}(\Delta_{t_2}\setminus\Delta_{t_1})\leq c (t_2-t_1+t_2^\lambda)t_2^\lambda
\end{equation}
for $t_2>t_1\geq 0$. In order to convert this into an estimate on $\mvol{2}(H_{t_2}\setminus H_{t_1})$, where as before $H_t=b^{-1}([-t,0])$, we will use the following elementary lemma.

\begin{lemma}\label{lem:inclusions}
  Let $l:\mathrealpos\to\mathrealpos$ be defined by
\[l(t)=l(\gamma_1(t))+l(\gamma_2(t)).\]
If $t\in\mathrealpos$ and $a\in (0,t)$ satisfy $l(t-a)\leq 2a$ and $l(t+a)\leq 2a$, then
\[\Delta_{t-a}\subseteq H_t\subseteq\Delta_{t+a}.\]
\end{lemma}

\begin{remark}
  If $l(t)\leq t^\lambda$ for all $t\geq t_0$, and if $\lambda<1$, then we have $l(t-t^\lambda)\leq 2t^\lambda$ and $l(t+t^\lambda)\leq 2t^\lambda$, and hence $\Delta_{t-t^\lambda}\subseteq H_t\subseteq\Delta_{t+t^\lambda}$, provided $t-t^\lambda\geq t_0$.
\end{remark}

We will continue the sketch of the proof of the theorem, and prove the lemma afterwards. Using Lemma \ref{lem:inclusions}, inequality \eqref{eqn:volumeDelta} and (ii'), we can find an upper bound on $\mvol{2}(H_{t_2}\setminus H_{t_1})$ of the type
  \begin{equation}\label{eqn:volumeH}
     \mvol{2}(H_{t_2}\setminus H_{t_1})\leq c(t_2-t_1+t_2^\lambda)t_2^\lambda 
  \end{equation}
if $t_2>t_1$ and $t_1$ is sufficiently large. By the coarea formula, cf.\ Corollary \ref{cor:diffinequality}, we have 
\begin{equation}
  \mvol{2}(H_{t_2}\setminus H_{t_1})=v(t_2)-v(t_1)=\int_{t_1}^{t_2} h(t)\, dt,
\end{equation}
where $h(t)=\mhdf^1(h_t)=\mhdf^1(h^{\gamma_1}_t)+\mhdf^1(h^{\gamma_2}_t)$. So \eqref{eqn:volumeH} provides a bound for the integrals of $h$ only. This does not suffice to prove directly our claim that $\lim\frac{1}{t}h(t)=0$. Additionally, we will use estimate (i') on the Gaussian curvature $K$ and Lemma \ref{lem:isoperimetric} to bound the variation of $h$. By Lemma \ref{lem:isoperimetric} we have for $t_2\geq t_1\geq 0$
\begin{equation}\label{eqn:isoperimetric2}
  h(t_2)-h(t_1)\geq -\int_{t_1}^{t_2}\omega(t)\, dt,
\end{equation}
where, as before, $\omega(t)=\int_{H_t}K\,d\mvol{2}$. So, we need an estimate on $\omega(t)$. Since the subcylinder $\Delta_t$ is bounded by two geodesic loops, we have
\[\left| \int_{\Delta_t} K\, d\mvol{2} \right|\leq 2\pi.\]
Now we can use Lemma \ref{lem:inclusions}, inequality \eqref{eqn:volumeDelta}, (i') and (ii') to show that
\begin{equation}\label{eqn:omegabound}
  \omega(t)< c\,t^{\kappa+2\lambda}+2\pi
\end{equation}
if $t$ is sufficiently large. Then we use \eqref{eqn:volumeH} - \eqref{eqn:isoperimetric2} for $t_2=t_1+t_1^{-\frac{\kappa}{2}}$ and combine this with \eqref{eqn:omegabound} to obtain an estimate of the type
\[h(t)\leq c\,t^{2\lambda+\frac{\kappa}{2}}.\]
Since we assumed that $2\lambda+\frac{\kappa}{2}<1$ this implies $\lim_{t\to\infty}\frac{1}{t}h(t)=0$. Now Theorem \ref{thm:premain} shows that $g$ is flat.
\begin{proof}[Proof of Lemma \ref{lem:inclusions}]
  We first show that $l(t-a)\leq 2a$ implies $\Delta_{t-a}\subseteq H_t$. Recall that $\partial\Delta_{t-a}$ consists of the shortest noncontractible loops based at $\gamma_1(t-a)$ and at $\gamma_2(t-a)$, both of length $\leq 2a$ by assumption. Since $\{\gamma_1(t-a),\gamma_2(t-a)\}\subseteq H_{t-a}=b^{-1}([-t+a,0])$ and since $b$ is $1$-Lipschitz we conclude that $\partial\Delta_{t-a}\subseteq H_t=b^{-1}([-t,0])$. Since $\Delta_{t-a}$ and $H_t$ are compact subcylinders of $C$ and $\partial\Delta_{t-a}\subseteq H_t$, we see that $\Delta_{t-a}\subseteq H_t$. To prove the inclusion $H_t\subseteq\Delta_{t+a}$ we use the same arguments as above to conclude that $H_t\cap\partial\Delta_{t+a}=\emptyset$. This implies that either $H_t\cap\Delta_{t+a}=\emptyset$ or $H_t\subseteq\Delta_{t+a}$. Since $\{\gamma_1(0),\gamma_2(0)\}\subseteq H_t\cap\Delta_{t+a}$, the second alternative has to be true.
\end{proof}

\begin{proof}[Finally, we give the details of the proof of Theorem \ref{thm:main}]
  In the following, the values of the constants in the inequalities have no significance. In order to simplify the estimates we will slightly increase $\kappa$ and $\lambda$, so that in addition to $2\lambda+\frac{\kappa}{2}<1$ the following holds:

There exists $t_0>1$ such that
\begin{equation}\label{eqn:boundl}
  l(t)\leq t^\lambda,\quad\textnormal{if } t\geq t_0,
\end{equation}
and such that
\begin{equation}\label{eqn:boundK}
  K(p)\geq -t^\kappa,\quad\textnormal{if } t\geq t_0\textnormal{ and }\dist(p,\Delta_0)\leq t.
\end{equation}

We choose $t_1>t_0$ such that $t_1-t_1^\lambda\geq t_0$. Then Lemma \ref{lem:inclusions} implies that
\begin{equation}\label{eqn:inclusions8}
  \Delta_{t-t^\lambda}\subseteq H_t\subseteq \Delta_{t+t^\lambda}
\end{equation}
for all $t\geq t_1$. If $t>s\geq t_1$ we conclude that
\begin{equation}\label{eqn:inclusions9}
  H_t\setminus H_s\subseteq \Delta_{t+t^\lambda}\setminus\Delta_{s-s^\lambda}.
\end{equation}
Next we will use Proposition \ref{prp:santalo2} to estimate $\mvol{2}(\Delta_t\setminus\Delta_s)$ for $t>s\geq 0$. The set $\Delta_t\setminus\Delta_s$ consists of the two subcylinders bounded by the shortest noncontractible loops based at the points $\gamma_1(s)$ and $\gamma_1(t)$, and at the points $\gamma_2(s)$ and $\gamma_2(t)$, respectively. By Proposition \ref{prp:santalo2} we can estimate the area of each of these subcylinders by $\frac{4}{\pi}\left( t-s+\frac{1}{2}(l(s)+l(t)\right)\left( l(s)+l(t)\right)$. Combining this with \eqref{eqn:boundl} and \eqref{eqn:inclusions9} we obtain for $t>s\geq t_1$:
\begin{eqnarray}\label{eqn:proof10}
  \mvol{2}(H_t\setminus H_s)
\leq\mvol{2}(\Delta_{t+t^\lambda}\setminus\Delta_{s-s^\lambda})
&\leq&\frac{16}{\pi}\left((t+t^\lambda)-(s-s^\lambda)+t^\lambda\right)t^\lambda\nonumber \\
&\leq& 16(t-s+t^\lambda)t^\lambda. 
\end{eqnarray}

Next we derive an upper estimate for $\omega(t)=\omega(H_t)=\int_{H_t}K\,d \mvol{2}$. If $t\geq t_1$, we can use \eqref{eqn:inclusions8} and obtain
\begin{equation}\label{eqn:proof11}
  \omega(H_t)=\omega(\Delta_{t+t^\lambda})-\omega(\Delta_{t+t^\lambda}\setminus H_t).
\end{equation}
Since $\Delta_{t+t^\lambda}$ is a subcylinder bounded by two geodesic loops we have
\begin{equation}\label{eqn:proof12}
  |\omega(\Delta_{t+t^\lambda})|<2\pi
\end{equation}
by the Gauss-Bonnet formula. Now we want to use the curvature estimate \eqref{eqn:boundK} to estimate $\omega(\Delta_{t+t^\lambda}\setminus H_t)$. To apply \eqref{eqn:boundK} we need to bound $\dist (p,\Delta_0)$ for $p\in\Delta_{t+t^\lambda}\setminus\Delta_0$.  Each $p\in \Delta_{t+t^\lambda}\setminus\Delta_0$ lies on a unique shortest noncontractible loop based at $\gamma_i(s)$ for some $i\in\{1,2\}$ and some $s\in(0,t+t^\lambda]$. Since $\gamma_i(0)\in\Delta_0$ we have
\[\dist(p,\Delta_0)\leq t+t^\lambda+\frac{1}{2}l(s)\leq t+2t^\lambda.\]
So, \eqref{eqn:boundK}, \eqref{eqn:inclusions8} and \eqref{eqn:proof10} imply for $t\geq t_1$:
\begin{eqnarray}\label{eqn:proof13}
  \omega(\Delta_{t+t^\lambda}\setminus H_t)&\geq& -(t+2t^\lambda)^\kappa\,\mvol{2}(\Delta_{t+t^\lambda}\setminus H_t)\\
&\geq& -(t+2t^\lambda)^\kappa\,\mvol{2}(\Delta_{t+t^\lambda}\setminus \Delta_{t-t^\lambda})\nonumber\\
&\geq& -48\,t^{\kappa+2\lambda}.\nonumber
\end{eqnarray}
From \eqref{eqn:proof11} - \eqref{eqn:proof13} we obtain for $t\geq t_1$:
\begin{equation}\label{eqn:proof14}
  \omega(t)=\omega(H_t)\leq 2\pi+48\,t^{\kappa+2\lambda}.
\end{equation}
Actually, a similar argument applied to the equality
\begin{equation}\label{eqn:proof15}
  \omega(t)=\omega(H_t)=\omega(\Delta_{t-t^\lambda})+\omega(H_t\setminus\Delta_{t-t^\lambda})
\end{equation}
shows that also $\omega(t)\geq -2\pi-48\,t^{\kappa+2\lambda}$, so that in fact
\[|\omega(t)|\leq 2\pi+48\,t^{\kappa+2\lambda}.\]
But for our purpose \eqref{eqn:proof14} is sufficient. Similarly, if instead of $\eqref{eqn:boundK}$ we have
\begin{equation}\tag{8.7'}\label{eqn:boundK'}
  K(p)\leq t^\kappa,\quad\textnormal{if } t\geq t_0\textnormal{ and }\dist(p,\Delta_0)\leq t,
\end{equation}
then the same argument applied to \eqref{eqn:proof15} shows that \eqref{eqn:proof14} also holds in this case.

To finish the argument we consider an arbitrary $t\geq t_1$ and recall that
\[\mvol{2}(H_{t+t^{-\frac{\kappa}{2}}}\setminus H_t)=\int_t^{t+t^{-\frac{\kappa}{2}}}h(t')d{t'}\]
by Corollary \ref{cor:diffinequality}. Hence, by \eqref{eqn:proof10}, there exists $s\in[t,t+t^{-\frac{\kappa}{2}}]$ such that
\[h(s)\;\leq\; 16\,(t^{-\frac{\kappa}{2}}+(t+t^{-\frac{\kappa}{2}})^\lambda)\,(t+t^{-\frac{\kappa}{2}})^\lambda\,t^{\frac{\kappa}{2}}\;\leq\; 64\,(t^{-\frac{\kappa}{2}}+t^\lambda)\,t^{\lambda+\frac{\kappa}{2}}.\]
Now Lemma \ref{lem:isoperimetric}, i.e.\ \eqref{eqn:isoperimetric2}, implies that
\[h(s)-h(t)\geq -\int_t^s\omega(t')\, d{t'}.\]
Using \eqref{eqn:proof14} and $t\leq s\leq t+t^{-\frac{\kappa}{2}}\leq t+1$ we obtain
\[h(t)\leq h(s)+\int_t^s\omega(t')\, d{t'}\leq 64\,(t^{-\frac{\kappa}{2}}+t^\lambda)\,t^{\lambda+\frac{\kappa}{2}} +(2\pi+48\,(t+1)^{\kappa+2\lambda})t^{-\frac{\kappa}{2}}.\]
Since $2\lambda+\frac{\kappa}{2}<1$, this implies $\lim_{t\to\infty}\frac{1}{t}h(t)=0$. Since $h(t) = \mhdf^1(h^{\gamma_1}_t) + \mhdf^1(h^{\gamma_2}_t)$, cf.\ \eqref{eqn:hsum}, we see that the assumptions of Theorem \ref{thm:premain} are satisfied in our situation. Hence the Riemannian metric $g$ is flat.
\end{proof}

\section{Appendix}
\begin{proof}[Proof of Lemma \ref{lem:isoperimetric}]
Approximate the compact subcylinder $G\defeq H_{t_2}\subseteq C$, provided with the Riemannian distance $\dist$, by an adequate sequence of polyhedra $(G^i,\dist^i)$, as in \cite[3.1.1]{BZG}. Denote by
\begin{align*}
  P^i_t&=\{p\in G^i: \dist^i(p,\partial G^i)<t\}, &f^i(t)&=\mhdf^2_{\dist^i}(P^i_t),\\
  l^i_t&=\{p\in G^i: \dist^i(p,\partial G^i)=t\}, &l^i(t)&=\mhdf^1_{\dist^i}(l^i_t),\\
  r^i&=\sup\{\dist^i(p,\partial G^i):p\in G^i\}.
\end{align*}
Whenever we omit the index $i$, we shall mean the corresponding objects for $(G,d)$, e.g.\ $l_t=\{p\in G: \dist (p,\partial G)=t\}$, $l(t)=\mhdf^1_d(l_t)$. By Corollary \ref{cor:distancefunction} (ii) we have 
\[b(p)=\dist (p,\partial G)-t_2\]
for all $p\in G\setminus H_0$, and hence
\begin{equation}\label{eqn:lh}
  h_t=l_{t_2-t} \;\text{ and }\; H_t=G\setminus P_{t_2-t}
\end{equation}
for all $0\leq t\leq t_2$. The Gauss-Bonnet theorem \cite[2.1.5]{BZG} implies that we have
\begin{equation}\label{eqn:gaussbonnetP}
  \omega(t)+\omega(P_{t_2-t})=-\tau
\end{equation}
for all $t_1\leq t\leq t_2$. Here $\tau$ is the rotation of $\partial G$, i.e.\ the integral of the geodesic curvature of $\partial G$ with respect to the inner normal. Let $T\defeq t_2-t_1$. Then, by \eqref{eqn:lh} and \eqref{eqn:gaussbonnetP}, the following inequality is equivalent to inequality \eqref{eqn:isoperimetric}:
\begin{equation}\label{eqn:isoperimetric'}
  l(T)-l(0)\leq\int_0^T -\omega(P_t)-\tau \,dt.
\end{equation}

To prove this inequality, we use the approximation and the lemma of Burago and Zalgaller. The lemma states, that inequality \eqref{eqn:isoperimetric'} is true for the polyhedra $(G^i,d^i)$. It then remains to show that it is still true in the limit case $(G,d)$. We begin by recalling the lemma of Burago and Zalgaller:
\begin{lemma}[{\cite[3.2.3]{BZG}}]
  The function $f^i$ is continuous on $[0,r^i)$ and twice differentiable on $[0,r_i)\setminus A_i$, where $A_i\subseteq (0,r_i)$ is finite. Moreover, for $t\in [0,r_i)\setminus A_i$,
  \begin{equation}\label{eqn:burago}
    (f^i)'(t)=l^i(t),\; (f^i)''(t)\leq -\omega^i(\overline{P}^i_t)-\tau^i.
  \end{equation}
Here $\overline{P}^i_t=P^i_t\cup l^i_t$, $l^i(t)=\mhdf^1_{d^i}(l^i_t)$ is the length of $l^i_t$, and $\tau^i$ is the rotation of $\partial G^i$. At the singular points $t\in A_i$ limits (from both sides) of the first derivative exist, and we have
\begin{equation}\label{eqn:jump}
  (f^i)'(t-0)\geq (f^i)'(t+0).
\end{equation}
\end{lemma}

 Integrating inequality \eqref{eqn:burago} and taking \eqref{eqn:jump} into account, we obtain
\begin{equation}
  \label{eqn:ineq1}
  l^i(b)-l^i(a)\leq\int_a^b -\omega^i(\overline{P}^i_t)-\tau^i\, dt,
\end{equation}
for $0\leq a\leq b\leq T$ and every $i\in\mathnatural$. For $[a,b]=[0,T]$ this implies that \eqref{eqn:isoperimetric'} is true for the polyhedra $(G^i,\dist^i)$. Next we show that we obtain inequality \eqref{eqn:isoperimetric'} by taking the limit $i\to\infty$ in \eqref{eqn:ineq1}. 

We will need that $\mhdf^2_{\dist^i}(P^i_t)\to \mhdf^2_\dist(P_t)$, i.e.\ $f^i(t)\to f(t)$, and $\omega^i(\closure{P}^i_t)\to\omega(P_t)$, for all $t\in[0,T]$. The proofs are given below. Observe that $\tau^i\to\tau$ by definition,
%, see \cite[2.1.5]{BZG} and \cite[p.\ 186]{AZI}}
 and that we have the weak convergence $(\omega^i)^\pm\to\omega^\pm$ for the positive and the negative parts of the curvature measures by \cite[3.1.1.(5)]{BZG}. Thus the integrands on the right hand side of \eqref{eqn:ineq1} are uniformly bounded by some constant $M<\infty$ for every $i\in\mathnatural$ and every $t\in[0,T]$. Lebesgue's theorem on dominated convergence implies that the right hand side of \eqref{eqn:ineq1} for $[a,b]=[0,T]$ converges to the right hand side of \eqref{eqn:isoperimetric'}. We also see from \eqref{eqn:ineq1} that $l^i(b)\leq l^i(a)+M\,(b-a)$ for every $i\in\mathnatural$ and $[a,b]\subseteq[0,T]$. Now we use the continuity of $l$, cf.\ \eqref{eqn:lh} and Corollary \ref{cor:diffinequality}, the coarea formula for $f$ and the $f^i$, and the pointwise convergence $f^i\to f$. Then we obtain
\begin{eqnarray*}
  l(t)&=&\lim_{\delta\downarrow 0}\frac{1}{\delta}(f(t+\delta)-f(t))\;=\;\lim_{\delta\downarrow 0}\lim_{i\to\infty}\frac{1}{\delta}(f^i(t+\delta)-f^i(t))\\
&=&\lim_{\delta\downarrow 0}\lim_{i\to\infty}\frac{1}{\delta}\int_t^{t+\delta}l^i\,d\mlge\;\leq\;\lim_{\delta\downarrow 0}\,\liminf_{i\to\infty}(l^i(t)+M\delta)\;=\;\liminf_{i\to\infty}l^i(t)
\end{eqnarray*}
for every $t\in[0,T]$. Together with the convergence $l^i(0)\to l(0)$, which holds by \cite[3.1.1.(4)]{BZG}, this proves the claim.

It remains to show that $f^i(t)\to f(t)$ and $\omega^i(\closure{P}^i_t)\to\omega(P_t)$ for every $t\in[0,T]$. Let $t\in[0,T]$ and write
\begin{equation}\label{eqn:fisifs}
  f^i(t)-f(t)=(\mhdf^2_{\dist_i}(P^i_t)-\mhdf^2_{\dist_i}(P_t)) + (\mhdf^2_{\dist_i}(P_t)-\mhdf^2_\dist(P_t)).
\end{equation}
By \cite[{Chapter VIII, Theorem 9}]{AZI}, we have that within any compact subset $Q\subseteq G$, the areas $\mhdf^2_{\dist^i}$ converge weakly to the area $\mhdf^2_\dist$. Since $\partial P_t=l_0\cup l_t$, we have $\mhdf^2_d(\partial P_t)=0$. Thus the second summand approaches zero, see e.g.\ \cite[VIII.4.10]{JEM}. Since the boundary curves $\partial G^i$ converge to $\partial G$ in the Hausdorff sense, and by the uniform convergence of the metrics $\dist^i\to \dist$, cf.\ \cite[3.1.1.(1)]{BZG}, we have that the distance functions $\dist^i(\cdot,\partial G^i)$ converge uniformly to $\dist (\cdot,\partial G)$. Hence, for large $i\in\mathnatural$, the symmetric difference $P_t^i \Delta P_t$ is contained in an arbitrarily small strip about $l_0$ and $l_t$. To be precise, let $\epsilon>0$, and, by continuity of $f$, choose $\delta>0$ such that the set $S_\epsilon\defeq\bigcup_{s\in[0,\delta]}l_s\cup\bigcup_{s\in[t-\delta,t+\delta]}l_s$ has area $\mhdf^2_\dist(S_\epsilon)<\epsilon$. Since $P^i_t\Delta P_t\subseteq S_\epsilon$ for almost all $i\in\mathnatural$ and by the weak convergence $\mhdf^2_{\dist_i}\to \mhdf^2_\dist$ again, we get $\mhdf^2_{\dist_i}(P^i_t\Delta P_t)\leq \mhdf^2_{\dist_i}(S_\epsilon)\leq \mhdf^2_\dist(S_\epsilon)+\epsilon\leq 2\epsilon$ for almost all $i\in\mathnatural$. This implies that the first summand on the right hand side of \eqref{eqn:fisifs} approaches zero, too. Hence $f^i(t)\to f(t)$ for every $t\in[0,T]$. The proof of the second claim, i.e.\ $\omega^i(\closure{P}^i_t)\to\omega(P_t)$ for all $t\in[0,T]$, is analogous, if one uses $(\omega^i)^\pm\to\omega^\pm$, cf.\ \cite[3.1.1.(5)]{BZG}.
\end{proof}

\section*{Acknowledgment}
This research was supported by the DFG Collaborative Research Center SFB~TR~71. The authors thank K. Burns and U. Lang for helpful discussions and the referee for careful reading.

{\noindent{}Mathematisches Institut\\
Universit\"at Freiburg\\
Eckerstr.~1\\
79104 Freiburg\\
Germany\\
\email{\textit{E-mail address}: {\tt victor.bangert@math.uni-freiburg.de}}\\
\email{\textit{E-mail address}: {\tt patrick.emmerich@math.uni-freiburg.de}}}


\begin{thebibliography}{10}

\bibitem{AZI}
A.~D. Aleksandrov and V.~A. Zalgaller, \emph{Intrinsic geometry of surfaces},
  Translations of Mathematical
  Monographs, Vol. 15, American Mathematical Society, Providence, R.I., 1967.

\bibitem{VBT}
V. Bangert, \emph{Totally convex sets in complete {R}iemannian manifolds},
  J. Differential Geom. \textbf{16} (1981), no.~2, 333--345.

\bibitem{ABO}
A.~S. Besicovitch, \emph{On two problems of {L}oewner}, J. London Math. Soc.
  \textbf{27} (1952), 141--144.

\bibitem{MBC}
M. Bialy, \emph{Convex billiards and a theorem by {E}. {H}opf}, Math. Z.
  \textbf{214} (1993), no.~1, 147--154.

\bibitem{DBA}
D.~D. Bleecker, \emph{The {G}auss-{B}onnet inequality and almost-geodesic
  loops}, Advances in Math. \textbf{14} (1974), 183--193.

\bibitem{GBI}
G.~Bol, \emph{Isoperimetrische {U}ngleichungen f\"ur {B}ereiche auf
  {F}l\"achen}, Jber. Deutsch. Math. Verein. \textbf{51} (1941), 219--257.

\bibitem{BIT}
D.~Burago and S.~Ivanov, \emph{Riemannian tori without conjugate points are
  flat}, Geom. Funct. Anal. \textbf{4} (1994), no.~3, 259--269.

\bibitem{BZG}
Yu.~D. Burago and V.~A. Zalgaller, \emph{Geometric inequalities}, Grundlehren
  der Mathematischen Wissenschaften, Band 285, Springer-Verlag, Berlin, 1988.

\bibitem{BKR}
K. Burns and G. Knieper, \emph{Rigidity of surfaces with no conjugate
  points}, J. Differential Geom. \textbf{34} (1991), no.~3, 623--650.

\bibitem{HBT}
H. Busemann, \emph{The geometry of geodesics}, Academic Press Inc., New
  York, N. Y., 1955.

\bibitem{CKT}
C.~B. Croke and B. Kleiner, \emph{On tori without conjugate points}, Invent. Math. \textbf{120} (1995), no.~2, 241--257.

\bibitem{JEM}
J. Elstrodt, \emph{Ma\ss- und {I}ntegrationstheorie}, dritte Auf\/l.,
  Springer-Lehr\-buch. Springer-Verlag, Berlin, 2002,
  Grundwissen Mathematik.

\bibitem{GMT}
H. Federer, \emph{Geometric measure theory}, Grundlehren der
  mathemati\-schen Wissenschaften, Band 153, Springer-Verlag New York Inc., New
  York, 1969.

\bibitem{FFL}
F.~Fiala, \emph{Le probl\`eme des isop\'erim\`etres sur les surfaces ouvertes
  \`a courbure positive}, Comment. Math. Helv. \textbf{13} (1941), 293--346.

\bibitem{LGS}
L.~W. Green, \emph{Surfaces without conjugate points}, Trans. Amer. Math. Soc.
  \textbf{76} (1954), 529--546.

\bibitem{KGC}
K. Grove, \emph{Critical point theory for distance functions},
  Differential geometry: {R}iemannian geometry ({L}os {A}ngeles, {CA}, 1990),
  Proc. Sympos. Pure Math., vol.~54, Amer. Math. Soc., Providence, RI, 1993,
  357--385.

\bibitem{EHC}
E. Hopf, \emph{Closed surfaces without conjugate points}, Proc. Nat.
  Acad. Sci. U. S. A. \textbf{34} (1948), 47--51.

\bibitem{HKR}
H. Koehler, \emph{Rigidity of cylinders without conjugate points}, Asian J.
  Math. \textbf{12} (2008), no.~1, 35--45.

\bibitem{PLL}
P. L\'evy, \emph{Le\c cons d'analyse fonctionelle}, Gauthier-Villars, Paris,
  1922.

\bibitem{PPR}
P. Petersen, \emph{Riemannian geometry}, Graduate Texts in Mathematics, vol.
  171, Springer-Verlag, New York, 1998.

\bibitem{SST}
K. Shiohama, T. Shioya and M. Tanaka, \emph{The geometry of
  total curvature on complete open surfaces}, Cambridge Tracts in Mathematics,
  vol. 159, Cambridge University Press, Cambridge, 2003.

\end{thebibliography}
\end{document}